\documentclass[12pt]{article}
\usepackage{graphicx} 
\usepackage{epstopdf}
\usepackage{subfigure}
\usepackage{color}
\usepackage[english]{babel}
\usepackage{amsmath,amsthm}
\usepackage{amsfonts}
\usepackage{booktabs}

\usepackage{graphicx}
\usepackage{subfigure}
\usepackage{tabularx}
\newtheorem{thm}{Theorem}[section]

\newtheorem{lem}[thm]{Lemma}

\title{Matching Book Embedding of the Cartesian Product of a Complete Graph and a Cycle}

\author { Zeling Shao, Yanqing Liu, Zhiguo Li{$^*$}\\
{\small School of Science, Hebei University of Technology, Tianjin 300401, China}
\date{}
\footnote{Corresponding author. E-mail: zhiguolee@hebut.edu.cn}
\footnote{This work is supported by the Natural Science Foundation of Hebei Province (No. A2019402043), China. }
}

\begin{document}
\baselineskip 0.65cm

\maketitle

\begin{abstract}
  The \emph{book embedding} of a graph $G$ is to place the vertices of $G$ on the spine and draw the edges to the pages so that the edges in the same page do not cross with each other. The book embedding is \emph{matching}
  if the pages have maximum degree $1$. The \emph{matching book thickness} is the minimum number of pages in which graphs can be matching book embedded.
  In this paper,  we show that   the matching book thickness of the Cartesian product $K_p\Box C_q$of a complete graph $K_p$ and a cycle $C_q$ is equal to $\Delta(K_p\Box C_q)+1$.


\noindent\textbf{2000 MR Subject Classification.} 05C10
\end{abstract}

\section{Introduction}

The concept of a book-embedding of a graph was introduced by Ollmann and Kainen.
A $book$ of the book embedding consists of a $spine$ which is just a line
and the $pages$ each of which is a half-plane with spine as boundary.
The $book~ thickness$   $bt(G)$  is a measure of the quality of a book embedding which is the minimum number of pages in which $G$ can be embedded.   Book embedding has been found wide applications in the computer science, VLSI theory, multilayer printed circuit boards, sorting with parallel stacks and turning-machine graphs (see $[1,2]$).
The Cartesian product of two arbitrary graphs $G$ and $B$ is the graph denoted by $G\Box B$ whose vertex set is $V(G)\times V(B)$, the vertex $(u_1, v_1)$ and the vertex $(u_{2}, v_{2})$ are adjacent in $G\Box B$ if and only if $u_1= u_2$ and $v_1$ is adjacent to $v_2$ in $B$, or $v_1= v_2$ and $u_1$ is adjacent to $u_2$ in $G$ (see $[3]$).
 The book embedding of graphs has been discussed for a variety of graph families, for examples
on complete graphs[4], complete bipartite graphs[5], generalized Petersen graph[6], some Cartesian product graphs[1,7,8], some semistrong product graphs[9] and Schrijver graphs[10], etc.

A book embedding of graph $G$ is $matching$ if every vertex must have degree at most one on every page.
The $matching~ book~ thickness$ of graph $G$, denoted by $mbt(G)$, is defined analogously to the book thickness as the minimum number of pages required by any matching book embedding of graph $G$ (see $[11-13]$).
A graph $G$ is $dispersable$ if  $mbt(G)=\Delta (G)$,
where $\Delta(G)$ is the maximum degree index (see $[4,11,13]$). 
 Complete bipartite graphs $K_{n,n}~(n\geq 1)$, even cycles $C_{2m}~(m\geq 2)$, cubes $Q_d~(d\geq 0)$ and trees are dispersable [4].  Given an arbitrary graph $G=(V,E)$, let $\chi^{'}(G)$ is the edge chromatic number,
by definition, it is easy to know that $mbt (G)\geq \chi^{'}(G)\geq \Delta (G)$.
Kainen $[14]$ showed that $mbt(C_{p}\Box C_{q})$ is $4$, when $p, q$ are both even and $mbt(C_{p}\Box C_{q})$ is $5$, when $p$ is even and $q$ is odd. But for the case  $p, q$ are both odd, $mbt(C_{p}\Box C_{q})$ is still undetermined.

 S. Overbay $[11]$ showed that any $k$-regular dispersable graph $G$ is bipartite. Also, she considered the matching book embedding of complete graphs $K_n$ and got $mbt(K_n)=n$.
Back in 1979, Bernhart and Kainen conjectured that any $k$-regular bipartite graph $G$ is dispersable, i.e., $mbt(G) = k$.
J.M. Alam, M.A. Bekos, M. Gronemann, M. Kaufmann, and S. Pupyrev $[13]$ disprove this conjecture for the cases $k = 3$, and
$k = 4$. In particular, they showed that
the Gray graph, which is $3$-regular and bipartite, has dispersable book
thickness four, while the Folkman graph, which is $4$-regular and bipartite, has dispersable book thickness five.

In this paper, we compute the matching book thickness of the Cartesian product of a complete graph $K_p$ and  a cycle $C_q$ as follows.

\textbf{Main Theorem:} For $p,q\geq 3,$   $mbt(K_p\Box C_q)=\Delta(K_p\Box C_q)+1$%

\section{Proof of the main theorem}

The proof will be completed by a sequence of lemmas.

\begin{lem}$^{[11]}$
If a graph $G$ is regular and dispersable, then $G$ is bipartite.
\end{lem}

\begin{lem}
Let $K_p$ be a complete graph and  $C_q$ be a cycle, $p,q\geq 3$, then $mbt(K_p\Box C_q)\geq \Delta(K_p\Box C_q)+1.$
\end{lem}

\begin{proof}
A graph is not bipartite if and only if it contains an odd cycle. Since the Cartesian product of $K_p$ and  $C_q$ contains at least one odd cycle,  then $K_p\Box C_q$ is not bipartite. Hence the graph $K_p\Box C_q$ is not dispersable.
Therefore $mbt(K_p\Box C_q)\geq \Delta(K_p\Box C_q)+1.$
\end{proof}

We will compute the matching book thickness of $K_p\Box C_q$ from four cases.

\noindent \textbf{Case 1:}~$K_{2n+1}\Box C_{2m+1}$

\begin{lem}
If $n\geq 2, m\geq 1$, then $mbt(K_{2n+1}\Box C_{2m+1})=\Delta(K_{2n+1}\Box C_{2m+1})+1=2n+3$.
\end{lem}

\begin{proof}
Since   $\Delta(K_{2n+1}\Box C_{2m+1})=2n+2$, by Lemma  $2.2$, we have $mbt(K_{2n+1}\Box C_{2m+1})\geq 2n+3$.

There are $(2n+1)(2m+1)$ vertices for graph $K_{2n+1}\Box C_{2m+1}$, and these vertices can be separated into 
$2n+1$ rows and $2m+1$ columns. Each columns holding a consecutive number from top to bottom. 
Let the ordering on spine be as $u_1v_{2n+1},u_1v_{2n},u_1v_{2n-1},...,u_1v_1,u_2v_1,u_2v_2,...,u_2v_{2n+1},...,\\u_{2m+1}v_{2n+1},...,u_{2m+1}v_1$.

Page $0$: the edge $\{(u_1v_1),(u_{2m+1}v_1)| m\geq 1\}$, and edges $\{(u_iv_a)(u_iv_b)|i=1$ or $2m+1; a+b=2n+3; 1< a,b\leq 2n+1\}$.

Page $1$: the edge $\{(u_1v_2),(u_{2m+1}v_2)|m\geq 1\}$, and edges $\{(u_iv_a),(u_iv_b)|i=1$ or $2m+1; a+b=2n+4; 2< a,b\leq 2n+1\}$.

Page $2$: the edge $\{(u_1v_3),(u_{2m+1}v_3)| m\geq 1\}$, and edges $\{(u_iv_a),(u_iv_b);(u_iv_1),(u_iv_2)|i=1$ or $2m+1; a+b=2n+5; 3< a,b\leq 2n+1\}$.

$......$

Page $n-1$: the edge $\{(u_1v_{n}),(u_{2m+1}v_{n})|n\geq 2, m\geq 1\}$, and edges $\{(u_iv_a),(u_iv_b);(u_iv_c),(u_iv_d)|\\ i=1$ or $2m+1; a+b=3n+2; c+d=n; n< a,b\leq 2n+1; 1\leq c,d< n\}$.

Page $n$: the edge $\{(u_1v_{n+1}),(u_{2m+1}v_{n+1})|n\geq 2, m\geq 1\}$, and edges $\{ (u_iv_a),(u_iv_b);(u_iv_c),(u_iv_d)| \\i=1$ or $2m+1; a+b=3n+3; c+d=n+1; n+1< a,b\leq 2n+1; 1\leq c,d< n+1\}$.

$......$

Page $2n-2$: the edge $\{(u_1v_{2n-1}),(u_{2m+1}v_{2n-1})|n\geq 2, m\geq 1\}$, and edges $\{(u_iv_a),(u_iv_b);(u_iv_{2n}),\\
(u_iv_{2n+1})|i=1$ or $2m+1; a+b=2n-1; 1\leq a,b< 2n-1\}$.

Page $2n-1$: the edge $\{(u_1v_{2n}),(u_{2m+1}v_{2n})|n\geq 2, m\geq 1\}$, and edges $\{(u_iv_a),(u_iv_b)|i=1$ or $2m+1; a+b=2n; 1\leq a,b< 2n\}$.

Page $2n$: the edge $\{(u_1v_{2n+1}),(u_{2m+1}v_{2n+1})| n\geq 2, m\geq 1\}$, and edges $\{(u_iv_a),(u_iv_b)|i=1$ or $2m+1; a+b=2n+1; 1\leq a,b< 2n+1\}$.

Page $2n+1$: edges $\{(u_iv_j),(u_{i+1}v_j)|1 \leq i< 2m+1; j=1,2,3,...,2n+1, i$ is odd$\}$, and edges $\{(u_{2m+1}v_a),(u_{2m+1}v_b)|a+b=2n+2; 1\leq a,b\leq 2n+1\}$.

Page $2n+2$: edges $\{(u_iv_j),(u_{i+1}v_j)| 2 \leq i< 2m+1; j=1,2,3,...,2n+1, i$ is even$\}$, and edges $\{ (u_1v_a),(u_1v_b)| a+b=2n+2; 1\leq a,b\leq 2n+1\}$.

For $k=0, 1, 2, 3,..., 2n$, 
other edges $\{(u_iv_a),(u_iv_b)|i= 2, 3, ..., 2m; 1\leq a,b\leq 2n+1  \}$ of $K_{2n+1}\Box C_{2m+1}$ are placed on page $k$, where $k\equiv a+b ~~($mod$~~ 2n+1)$. 

 So   $K_{2n+1}\Box C_{2m+1}$ can be matching book embedded in  $2n+3$ pages. Therefore the matching book thickness of $K_{2n+1}\Box C_{2m+1}$ is $2n+3$ (see Fig.1 for the case $n=2, m=1$).
\end{proof}

\begin{figure}[htbp]
\centering
\includegraphics[height=4cm, width=0.7\textwidth]{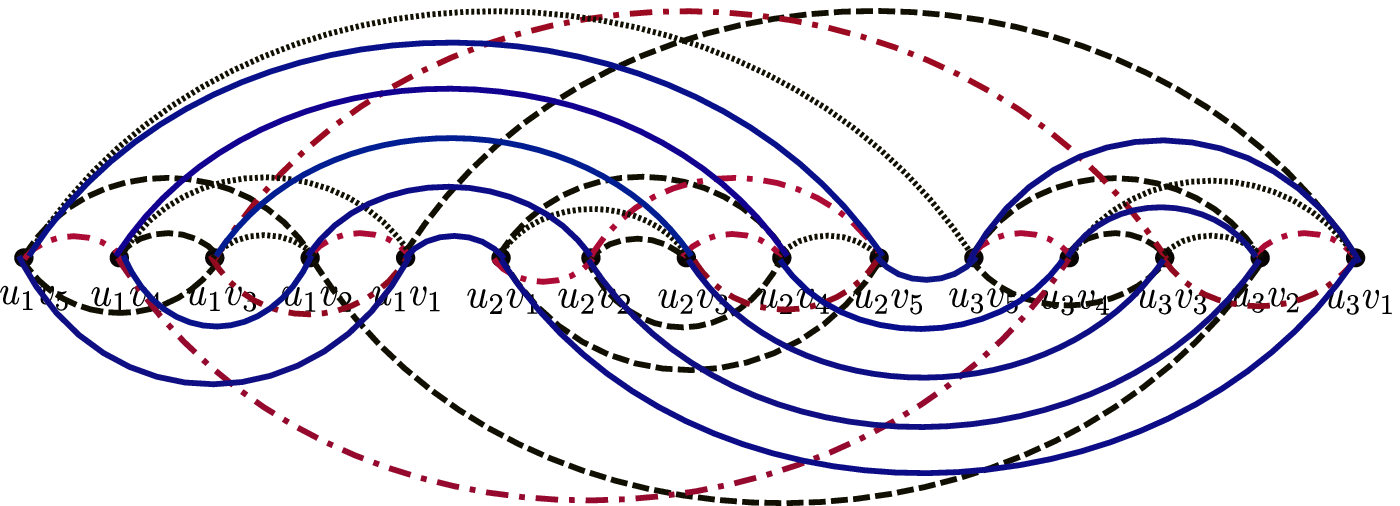}
\centerline{Fig.1  ~The matching book embedding of $K_{5}\Box C_{3}$.}%
\end{figure}

\noindent \textbf{Case 2:}~$K_{2n}\Box C_{2m+1}$

\begin{lem}
If $n\geq 2, m\geq 1$, then $mbt(K_{2n}\Box C_{2m+1})=\Delta(K_{2n}\Box C_{2m+1})+1=2n+2$.
\end{lem}

\begin{proof}
Since   $\Delta(K_{2n}\Box C_{2m+1})=2n+1$, by Lemma $2.2$, $mbt(K_{2n}\Box C_{2m+1})\geq 2n+2$.
Vertices of graph $K_{2n}\Box C_{2m+1}$ can be separated into
$2n$ rows and $2m+1$ columns. The edges of graph $K_{2n}\Box C_{2m+1}$ can be put on  $2n+2$ pages  analogously to that of $K_{2n+1}\Box C_{2m+1}$ in Lemma $2.3$. 
Hence we have $mbt(K_{2n}\Box C_{2m+1})\leq 2n+2$.  Therefore the result is established (see Fig.2 for the case $n=3, m=1$).

\end{proof}

\begin{figure}[htbp]
\centering
\includegraphics[height=4cm, width=0.7\textwidth]{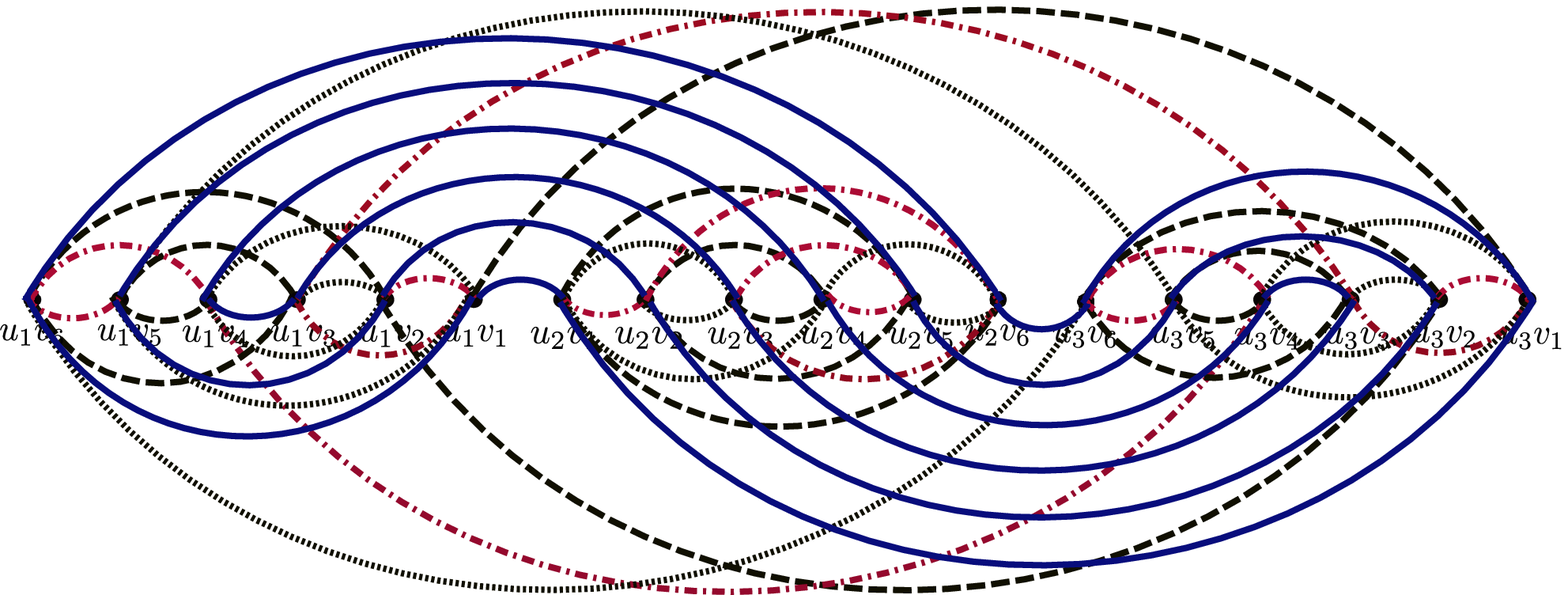}
\centerline{Fig.2 ~The matching book embedding of $K_{6}\Box C_{3}$.}
\end{figure}


 The following result will be applied to case 3 and case 4, which is a generalization of  Theorem 4.3 of  $[4]$:

\begin{lem} Let $G$ be arbitrary and $B$ be a dispersable bipartite graph, then $mbt (G\Box B)\leq mbt(G)+\Delta (B)$.
\end{lem}

\begin{proof}
Since $B$ is dispersable, then there is a $\Delta(B)$-edge coloring and a corresponding matching book embedding in a $\Delta(B)$-page book so that all edges of one color lie in the same page. Since $B$ is bipartite, there is a 2-vertex coloring using colors $a$ and $b$.

Now we consider the matching book embedding of $G\Box B$. Take a matching book embedding of $G$ in $mbt(G)$ pages. Using a matching book embedding of $B$, replace all vertices which colored $a$ with a copy of this matching book embedding of $G$ and replace all vertices which colored $b$ with the same matching book embedding of $G$, but in reverse order.

Now each of vertex of $B$ is represent by a copy of $G$.
Since each of these copies are placed on the spine, then the edges of $G\Box B$ corresponding to the copies of $G$ can all be matching embedded in $mbt(G)$ pages.

The remaining edges of $G\Box B$ connect corresponding vertices in adjacent copies of $G$. Since $B$ is bipartite, then edges of $B$ connect vertices of different colors. Since the order of vertices in adjacent copies of $G$ is reversed, then edges connect two adjacent copies corresponding to a single edge of $B$ can be embedded in the appropriate page as concentric semicircles, which makes sure that the book embedding is matching. 
 Since $mbt(B)=\Delta(B)$, the copies of the edges of $B$ can also embedded in $\Delta(B)$ pages.

Hence all edges of $G\Box B$ can be matching book embedded in  $mbt(G)+\Delta (B)$ pages.
\end{proof}

For example, let $G=K_5-e$, $B=P_3$,
the matching book embedding of $G$ and $B$ is shown in Fig.$3$.
The resulting matching book embedding of graph $G\Box B$ in $7$ pages is shown in Fig.$4$. 


\begin{figure}[htbp]
\centering
\includegraphics[height=3.4cm, width=0.7\textwidth]{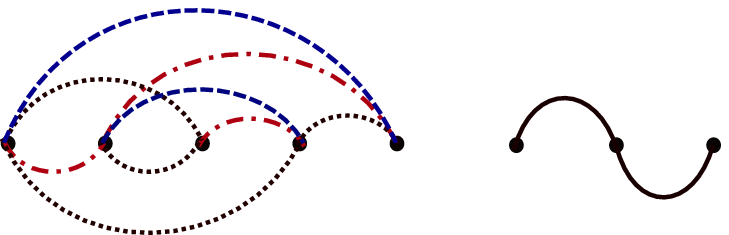}
\centerline{Fig.3 ~The matching book embedding of $G$ (left) and $B$(right)}
\end{figure}


\begin{figure}[htbp]
\centering
\includegraphics[height=4cm, width=0.7\textwidth]{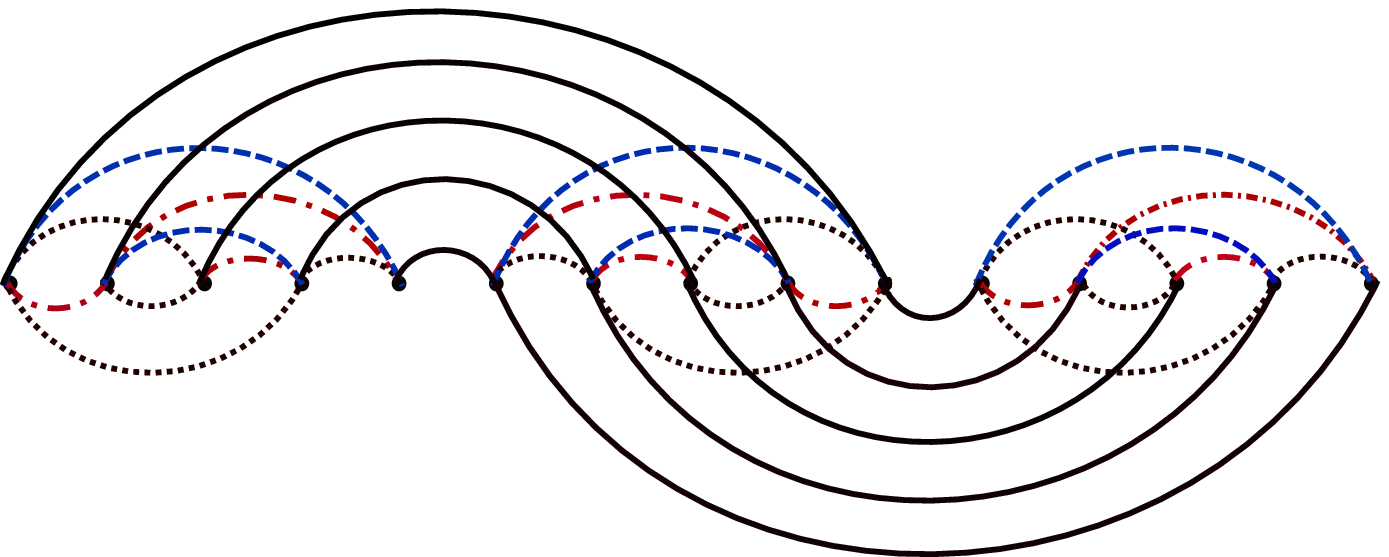}
\centerline{Fig.4 ~The matching book embedding of $G\Box B$}
\end{figure}


\noindent \textbf{Case 3 and Case 4:}~$K_{2n}\Box C_{2m}, K_{2n+1}\Box C_{2m}$
\begin{lem}
If $n, m\geq 2$, then\\
 (i) $mbt(K_{2n}\Box C_{2m})=\Delta(K_{2n}\Box C_{2m})+1=2n+2$;\\
 (ii) $mbt(K_{2n+1}\Box C_{2m})=\Delta(K_{2n+1}\Box C_{2m})+1= 2n+3$.
\end{lem}

\begin{proof}
Since $K_{2n}\Box C_{2m}$ is $2n+1$-regular and $K_{2n+1}\Box C_{2m}$ is $2n+2$-regular, by Lemma   $2.2$, it can be obtained that $mbt(K_{2n}\Box C_{2m})\geq 2n+2$, $mbt(K_{2n+1}\Box C_{2m})\geq 2n+3$.

While the even cycle $C_{2m}$ is dispersable, by Lemma 2.5, for complete graph $K_{2n}$ and $K_{2n+1}$, we have
  $$mbt(K_{2n}\Box C_{2m})\leq mbt(K_{2n} )+2=2n+2;$$ $$mbt(K_{2n+1}\Box C_{2m})\leq mbt(K_{2n+1})+2= 2n+3.$$
This gives the upper bound of $K_{2n}\Box C_{2m}$ and $K_{2n+1}\Box C_{2m}$. Hence the equalities of (i) and (ii) both hold.
\end{proof}


\end{document}